\documentclass{amsart}
  \usepackage{amscd,amssymb,epsfig, epsf}
   \usepackage{epic,eepic}

                     \numberwithin{equation}{subsection}

                     \newtheorem{propo}{Proposition}[section]
                     \newtheorem{corol}[propo]{Corollary}
                  \newtheorem{theor}[propo]{Theorem}
                     \newtheorem{lemma}[propo]{Lemma}
                     \theoremstyle{definition}

                     \theoremstyle{remark}

                     \newcommand{\RR}{\mathbb{R}}

              \newcommand{\card}{\operatorname{card}}

\begin{document}
      \title{Trimming  of  metric spaces and the tight span}
                     \author[Vladimir Turaev]{Vladimir Turaev}
                     \address{%
              Department of Mathematics, \newline
\indent  Indiana University \newline
                     \indent Bloomington IN47405 \newline
                     \indent USA \newline
\indent e-mail: vtouraev@indiana.edu} 
                    \begin{abstract}
We use the    trimming transformations to study the tight span of a metric space.
\end{abstract}

\maketitle

\section{Introduction}\label{intro}

 The  theory of tight spans due to  J.~Isbell \cite{Is} and A.~Dress \cite{Dr}   embeds any metric space~$X$   in a hyperconvex metric space $T(X)$   called the tight span of~$X$. In this paper we split  $T(X)$   as a  union of two metric subspaces   $\tau= \tau(X) $ and   $\overline {C}=\overline {C(X)} $.  The   space~$\tau  $      is the tight span of a certain  quotient $\overline{X_\infty}$ of~$X$. The   space~$\overline {C } $ is a disjoint union of metric trees which   either do not meet~$\tau  $ or  meet~$\tau  $   at their roots lying in $\overline{X_\infty}\subset \tau$.    In this picture, the original metric space $X\subset T(X)$   consists of the tips of the branches of the trees.


The construction of   $\tau  $ and~$\overline {C } $    uses    the  trimming transformations of metric spaces   studied   in \cite{Tu}   for finite metric spaces. In the present paper - essentially independent of~\cite{Tu} -  we discuss trimming for all metric spaces  and introduce    related objects including the subspaces $   \tau $  and~$ \overline C$ of $T(X)$. Our main theorem   says that $\tau  \cup \overline C=T(X)$ and   $\tau \cap{\overline C }$ is the set of the  roots of the trees forming~$\overline C$.

This work was partially supported by the NSF grant DMS-1664358​.

  \section{Trim  pseudometric spaces}\label{sectionBIS0}

 \subsection{Pseudometrics}\label{PreliminariesBIS01} We  recall basics on metric and  pseudometric spaces. Set $\RR_+=\{r\in \RR \, \vert \, r\geq 0\}$.
  A \emph{pseudometric space} is a pair   consisting of a set~$X$ and a mapping $d: X\times X \to \RR_+ $ (the \emph{pseudometric}) such that for all $x,y,z\in X$,   $$d(x,x)=0, \quad d(x,y)=d(y,x) , \quad d(x,y)+ d(y,z)\geq d(x,z).$$
  A pseudometric space $(X,d)$ is a \emph{metric space} (and~$d$ is a \emph{metric})  if $d(x,y) >0$ for all  distinct $x, y\in X $.

  A map $f:X\to X'$ between pseudometric spaces  $(X,d)$ and $(X',d')$  is \emph{distance preserving} if $d(x,y) = d'(f(x), f(y))$ for all $x,y \in X$ and is \emph{non-expansive} if $d(x,y) \geq d'(f(x), f(y))$ for all $x,y \in X$. Pseudometric spaces $X,X'$ are \emph{isometric} if there is a   distance preserving bijection $X \to X'$.  We call distance preserving maps between metric spaces   \emph{metric embeddings}; they are always injective.

Each pseudometric space $(X,d)$ carries an equivalence relation $\sim_d$  defined by $x\sim_d y$ if $d(x,y)=0$ for   $x,y\in X$.    The quotient set $ \overline X = X/{\sim_d }$  carries a metric~$\overline d$ defined by
$ {\overline d}(\overline x, \overline y)=d(x,y)$ where   $x,y$ are any points of~$X$ and $\overline x, \overline y \in  \overline X$ are their equivalence classes.  The metric space $(\overline X, \overline d)$ is   the  \emph{metric quotient} of   $(X,d)$. Any distance preserving map from $X$ to a metric space $Y$ expands uniquely as the composition of the projection $X\to \overline X$ and a  metric embedding   $\overline X\hookrightarrow Y$.

  \subsection{Trim spaces}\label{Preliminaries} Given a set~$X$ and a  map  $d:X\times X\to \RR$, we      use the same symbol~$d$ for the map  $X\times X\times X \to \RR $ carrying any triple $x,y,z \in X$ to
$$d(x,y,z)=  \frac{ d(x, y)+ d(x, z)-d(y,z)}{2} . $$
  The   right-hand side   is  called the Gromov product, see, for instance, \cite{BCK}.

A  pseudometric~$d$ in a set~$X$   determines   a function $\underline d: X \to \RR_+ $ as follows:
if $\card (X)=1$, then $\underline d =0$; if $X$ has two  points $x,y$, then $\underline d(x)=\underline d(y)= d(x,y)/2$; if $\card (X)\geq 3$, then for all $x \in X$,
$$ {\underline d} (x)= \inf_{y,z \in X\setminus \{x\}, y \neq z }   d(x, y,z)  . $$
If $\underline d(x)=0$ for all $x\in X$, then we say that the pseudometric space   $(X,d)$ is \emph{trim}.

Following K.~Menger \cite{Me}, we say that a point~$x$ of a pseudometric space $(X,d)$   \emph{lies  between} $y\in X$ and $z\in X$ if $d(y,z)=d(x,y)+d(x,z)$.
 A simple sufficient   condition for  $(X,d)$ to be trim says that each point  $x\in X$ lies between  two  distinct points of  $  X\setminus \{x\}$. 

\subsection{Examples}   
A pseudometric space   having only one point  is trim. A pseudometric space   having    two points   is trim iff the distance between these points is equal to zero.
More generally, any  set with  zero pseudometric   is trim.  We  give two   examples of trim   metric spaces:

  (a)  the set of words of a fixed finite length in a given finite  alphabet  with the Hamming distance between words defined as the number of positions at which the   letters of the words   differ;

(b) a  subset   of a  Euclidean circle $C\subset \RR^2 $   meeting each   half-circle in $C$   in at least three points; here the distance between two  points  is the length of the shorter arc  in $C$ connecting these points.

%
%
%
%

\section{Drift and trimming}

\subsection{The drift}\label{Drift} Given  a  function $\delta: X\to \RR$ on a     metric space $(X,d)$, we define a  map  $d_\delta:X\times X\to \RR$  as follows: for   $x,y \in X$,
  \begin{equation*}
    d_\delta(x,y)=\left\{
                \begin{array}{ll}
                0 \quad {\text {if}} \,\,  x=y , \\
                  d(x,y)-\delta(x)- \delta(y) \quad {\text {if}} \,\,   x \neq y   .
                \end{array}
              \right.
\end{equation*}
 We say that $d_\delta$ is obtained from~$d$ by  a \emph{drift}. The idea behind the definition of $d_\delta$  is that each  point $x\in X$ is pulled towards   all other points of~$X$ by  $\delta(x)$.


   \begin{lemma}\label{le49} If    $\delta \leq  {\underline d}$ (i.e., if
   $\delta(x)\leq  {\underline d}(x)$ for all $x\in X$), then $ {d_\delta} $ is a   pseudometric in~$X$. Moreover, if $\card(X) \geq 3$ or $\card(X) =2, \delta=const$, then   $\underline{d_\delta}= \underline d-\delta$.
\end{lemma}

\begin{proof} We first prove that for any distinct $x,y \in X$,
\begin{equation}\label{eqsimple} {\underline d} (x) + {\underline d} (y) \leq d(x,y).  \end{equation}
If $X=\{x,y\}$, then \eqref{eqsimple} follows   from  the definition of~$\underline d$. If $\card(X)\geq 3$,   pick any $z\in X\setminus \{x,y\}$. Then
${\underline d} (x) \leq   { d(x, y ,z)} $
and  ${\underline d} (y) \leq   d(y,x ,z) $.
So,
$${\underline d} (x) + {\underline d} (y) \leq d(x,y,z) + d(y,x,z)=d(x,y).$$

We   now check that   ${d_\delta}:X \times X\to \RR$ is a pseudometric. Clearly,   ${d_\delta}$ is symmetric and,  by   definition, ${d_\delta}(x,x)=0$ for all $x\in X$. Formula \eqref{eqsimple} and the assumption $\delta \leq  {\underline d}$ imply that for any distinct $x,y\in X$,
$${d_\delta}(x,y)=d(x,y)-\delta(x)- \delta(y)\geq {\underline d} (x)  -\delta(x)+  {\underline d} (y) - \delta(y) \geq 0.$$
To prove the triangle inequality for $d_\delta$ we rewrite it as $d_\delta(x,y,z)\geq 0$
  for any
$x,y,z\in X$. If $x=y$ or $x=z$, then   $d_\delta(x,y,z)  =0$; if $y=z$, then  $d_\delta(x,y,z)=d_\delta(x,y) \geq 0$. Finally, if $x,y,z$ are pairwise distinct, then
 \begin{equation}\label{ert} d_\delta(x,y,z)  =  d(x,y,z) -  \delta(x) \geq  {\underline d} (x) -  \delta(x) \geq 0  . \end{equation}

The equality $\underline{d_\delta} (x)= \underline d (x) -\delta (x)$  for all $x\in X$ follows   from \eqref{ert} if $\card(X)\geq 3$ and from the definitions
if $\card(X)=2$ and $\delta=const$.
\end{proof}

\subsection{Trimming}\label{Trimming} Let $(X,d)$ be a  metric space.
Applying  Lemma~\ref{le49} to  $\delta= \underline d:X\to \RR$, we obtain a pseudometric $d^{-}=d_\delta$   in~$X$. For $x,y \in X$,
  \begin{equation}\label{dbulletdef}
    d^{-}(x,y)=\left\{
                \begin{array}{ll}
                0 \quad {\text {if}} \,\,  x=y  , \\
                  d(x,y)- \underline d(x)- \underline d(y) \quad {\text {if}} \,\,   x \neq y.
                \end{array}
              \right.
\end{equation}
  Lemma~\ref{le49} implies that    $\underline{d^{-}}=0$, i.e.,    that   $(X, d^{-})$ is a trim pseudometric space.
Let $t(X)$ be   the metric quotient
      of   $(X, d^{-})$.  By definition, $t(X)=X/{\sim_{d^{-}}}$ and the metric in $t(X)$ is induced by the pseudometric $d^{-}$ in~$X$. We    say that    $t(X)$ is obtained from $(X,d)$ by \emph{trimming}  and call
  the projection $p_X:X  \to   t(X) $ the \emph{trimming projection}. Note that $p_X$ is  a   non-expansive surjection of metric spaces.  It   is bijective   if and only if    ${d^{-}}$   is a metric in~$X$ and then    $t(X)  =(X, d^{-}) $ is a trim metric space. In general,   $t(X)$   may be non-trim.

  Starting from $ (X_0,d_0)=(X,d)$ and iterating the trimming,     we  obtain   metric spaces  $\{  t^k(X)=(X_k, d_k)\}_{k\geq 1}$  and non-expansive surjections
   \begin{equation}\label{line} X =X_0   \stackrel{p_0}{\longrightarrow} X_1 \stackrel{p_1}{\longrightarrow} X_2  \stackrel{p_2}{\longrightarrow}  \cdots    \end{equation}  where   $p_0=p_X$ and     $p_k=p_{t^k(X)}: X_k \to X_{k+1}   $ for   $k\geq 1$. We call  the   sequence   \eqref{line} the \emph{trimming sequence} of~$X$.
Each point $x \in X $ determines    points   $(x_{(k)} \in X_k)_{k\geq 0}$   by   $x_{(0)}=x$ and $x_{(k+1)}=p_k(x_{(k)})$ for  all~$k$. The sequence $(x_{(k)}  )_{k\geq 0}$ is the \emph{trimming sequence} of~$x$.

   For   $x,y \in X$, we  write $x \cong y$ if there is an integer $k\geq 0$ such that $x_{(k)}=y_{(k)}$. We call the smallest such~$k$ the \emph{meeting index} of $x, y$ and denote it by $m(x,y)$.  Note that   $x_{(k)} =y_{(k)} $ for all $k \geq m(x,y)$ which easily implies that $\cong$ is an equivalence relation in~$X$. Clearly, $m(x,x)=0$   and $m(x,y) \geq 1$ for $x\neq y$.

    Set $X_\infty=X/\cong$ and for  $x\in X$, denote its projection to $X_\infty$ by $x_{(\infty)}$.  We have
 $$d(x,y) \geq d_1(x_{(1)} , y_{(1)} ) \geq d_2(x_{(2)} , y_{(2)} ) \geq \cdots \geq 0  $$
  for any $x,y \in X$. As a consequence, the formula
 $$d_\infty (x_{(\infty)},y_{(\infty)})= \lim_{k\to \infty}   d_k(x_{(k)}, y_{(k)}) $$
 defines a pseudometric $d_\infty$ in  $X_\infty$. 
   Clearly, the projection $ X\to X_\infty, x\mapsto x_{(\infty)}$ is a non-expansive surjection.


\subsection{Remark} One can take the metric quotient $\overline {X_\infty}$ of $X_\infty$ and then iterate the   construction $X\mapsto \overline {X_\infty}$. We will not do it in this paper.

\section{Leaf spaces of metric trees}\label{examplesfmt}

  We compute the trimming for the  leaf spaces  of    metric trees.

 \subsection{Metric  graphs and trees}\label{Trees}    A \emph{metric graph} is a   graph without loops  whose   every  edge~$e$ is endowed with a      real number $\ell_e > 0$, the  \emph{length},  and with a homeomorphism   $e \approx [0, \ell_e] \subset \RR_+ $.  We use this homeomorphism to define the length of any subsegment of~$e$: this is just the  length of the corresponding subsegment of  $[0, \ell_e]  $.  A wider class of \emph{pseudometric graphs}  is defined   similarly but allowing the lengths of  edges   to be equal to zero. Only the edges  of positive length carry a homeomorphism   $   e \approx [0, \ell_e]  $. All subsegments of edges of zero length   have zero length. 


A
\emph{pseudometric} (respectively, \emph{metric}) tree  is a pseudometric (respectively,  metric) graph which is path connected and has no cycles.  The underlying set of a  pseudometric tree~$L$  carries a \emph{path pseudometric} $d_L$: the distance $d_L(x,y) $ between any points $x,y  \in L$ is the sum of the lengths of the edges and subedges forming an  injective path from~$x$ to~$y$   in~$L$. The pseudometric $d_L$ is a metric   if and only if~$L$ is a metric tree. Every pseudometric tree~$L$ yields a metric tree~$\overline L$    by contracting each edge of~$L$ of zero length   to a point while keeping the homeomorphisms associated with the edges of positive length. The underlying metric space $({\overline L}, d_{\overline L})$ of~$\overline L$ is nothing but the metric quotient of   $(L,d_L)$.

 \subsection{The leaf space}\label{Leaf} A    \emph{leaf}   of a pseudometric tree~$L$  is a vertex of~$L$   adjacent to a single edge of~$L$.     The \emph{leaf space} $ {\partial L}$    is the set of all   leaves of~$L$      together with the     pseudometric~$d$ obtained by restricting $d_L$ to this set.
   The pseudometric space $(\partial L,d)$ depends only on the tree~$L$ and the lengths of its edges; it does not depend  on the choice of   homeomorphisms $   e \approx [0, \ell_e]  $.   
   The following lemma estimates (and in some cases computes) the function $\underline d:\partial L \to \RR_+$.


     \begin{lemma}\label{cloleththree--} Let $(\partial L,d)$ be  the leaf space of a  pseudometric tree~$L$ such that $\card (\partial L) \geq 3$.  Then for all   $x\in \partial L$, we have   ${\underline d}(x) \geq l_{e(x)} $  where $e(x) $ is the   edge of~$L$ adjacent to~$x$. If the second vertex of $e(x)$   is adjacent   to at least two   leaves   of $L$ besides~$x$, then ${\underline d}(x) = l_{e(x)} $.
  \end{lemma}

 \begin{proof} Let $v$ be the   vertex of the edge $e(x)$ distinct from~$x$.  Pick any   $y,z \in \partial L \setminus \{x\}$. It is clear that an injective path from $x$ to~$y$ passes through~$v$. Therefore
 $ d_L (x,y)=l_{e(x)} + d_L (v,y) $ and, similarly,  $d_L  (x,z)=l_{e(x)} + d_L (v,z) $.
 Then
  $$ d(  x,y,z)=  \frac{ d_L(x,y)+ d_L(x,z)-d_L(y,z)}{2} $$
  $$=l_{e(x)} +\frac{ d_L(v,y)+ d_L(v,z)-d_L(y,z)}{2} \geq l_{e(x)} .$$
  Since this holds for all $y,z $ as above, ${\underline d}(x) \geq l_{e(x)} $. Suppose now that~$v$   is adjacent   to  distinct   leaves  $y,z \in \partial L \setminus \{x\}$. Then the edges $e(x)$, $e(y)$ form an injective path from $x$ to~$y$ and so $d_L(x,y)= l_{e(x)} +l_{e(y)}$. Similarly, $d_L(x,z)= l_{e(x)} +l_{e(x)}$ and $d_L(y,z)= l_{e(y)} +l_{e(z)}$. Then $ d(  x,y,z)= l_{e(x)}$. Therefore ${\underline d}(x) = l_{e(x)} $.
 \end{proof}

 \subsection{Example} Let~$L$ be a metric tree whose every vertex  adjacent to a leaf is   adjacent to at least three leaves.  Lemma~\ref{cloleththree--}  implies that  $t(\partial L)$  is   the set of all vertices of~$L$ adjacent to a leaf with the metric   $d_L$ restricted to this set.



 \subsection{Example} \label{A  construction of rooted forests} Consider   an infinite  sequence of sets and surjective maps
 \begin{equation}\label{line++} X_0   \stackrel{p_0}{\longrightarrow} X_1 \stackrel{p_1}{\longrightarrow} X_2 \stackrel{p_2}{\longrightarrow}  \cdots  .  \end{equation}
 Fix any  functions $\{\delta_k :X_k\to (0, \infty)\}_{k\geq 0} $. Consider the metric graph~$L$ with the set of vertices   $\amalg_{k\geq 0}  X_k$ and with each      $v\in X_k$    connected to   $p_k(v)\in X_{k+1}$ by an edge  of length $ {\delta_k} (v)$. Assume   that    $\card (p_{k}^{-1} (a)) \geq 3$ for  all $k\geq 0$ and $a\in X_{k+1 }$   and that for any distinct $x,y \in X$ there is an integer $m=m(x,y)\geq 1$ such that $$(p_m \cdots p_1 p_0)(x)= (p_m \cdots p_1 p_0)(y) .$$
 Then $L$ is a metric tree. For all $k\geq 0$, restricting the  tree  metric $d_L$   to $X_k\subset L$, we obtain a  metric space   $ (X_k, d_k)$ where   $d_k$ is  computed by
 $$d_k(x,y)= \sum_{i=k}^{m-1} \big ( \delta_i((p_{i-1} \cdots p_k)(x)) +\delta_i((p_{i-1} \cdots p_k)(y)) \big)$$ for any  distinct   $x,y\in X_k$, where $m>k$ is the smallest integer such that
 $$(p_{m-1} \cdots p_k)(x)  = (p_{m-1} \cdots p_k)(y).$$
 Lemma~\ref{cloleththree--} implies that   $ \underline {d_k}=\delta_k$ and
   $(X_{k+1}, d_{k+1})$ is obtained from $(X_k, d_k)$ by trimming: $ (X_{k+1}, d_{k+1})=t(X_k, d_k)$.
   Then the sequence \eqref{line++} is    the trimming sequence of $\partial L=(X_0,d_0)$ and $X_\infty$ is a singleton (i.e.,  $\card (X_\infty)=1$).

\subsection{Example} \label{A ++ construction of rooted forests} We describe sets and   maps satisfying the conditions of Example~\ref{A  construction of rooted forests}. Let~$A$ be a set with $ \geq 3$ elements and let  $a\in A$. For  $k\geq 0$, let $X_k$ be the set of all infinite sequences $a_0,a_1,...$ of elements of $A$ such that $a_n=a$ for all sufficiently big~$n$. The map $p_k\colon X_k \to X_{k+1}$ drops the first element of a sequence. Given    functions $\{\delta_k :X_k\to (0, \infty)\}_{k\geq 0} $,   Example~\ref{A  construction of rooted forests} provides, for each $k\geq 0$, a metric $d_k$ in $X_k$ such that $ (X_{k+1}, d_{k+1})=t(X_k, d_k)$ for all~$k$.

 \section{The trimming cylinder}\label{trimmcylinder}

 We introduce   the   trimming cylinder   of a metric space. 

  \subsection{The cylinder~$C$}  Consider a   metric space $(X,d)=(X_0, d_0)$ and its trimming sequence \eqref{line} where   $ t^k(X)=(X_k, d_k)$ for all $k\geq 1$. We define a   graph $C=C(X)$: take the disjoint union $\amalg_{k\geq 0}  X_k$ as the set of vertices and connect each     $v\in X_k$    to    $p_k(v)\in X_{k+1}$ by an edge $e_v$ of length $ {\underline {d_k} } (v)\geq 0$.  If $  {\underline {d_k} } (v)> 0$, then we take for $e_v$ a copy of the segment $[0,  {\underline {d_k} } (v) ]\subset \RR_+$. Clearly, $C$ is a pseudometric graph and   $X =X_0$ is the set of its leaves. We call~$C$   the \emph{trimming cylinder} of~$X$.

  It is clear that all (path connected) components of $C$ are trees. Two points $x,y\in X \subset C$ lie in the same   component of~$C$ if and only if $x \cong y$. We can therefore identify the set  $\pi_0(C)$ of  components of~$C$ with $X_\infty=X/\cong$.


 Each  component  of~$C$, being a pseudometric tree,  carries the  path pseudometric. Also, the
  set    $X_k \subset C$     carries the  metric $d_k$ for all $k\geq 0$.   The following theorem  extends all these pseudometrics and metrics to a  pseudometric  in~$C$.

 \begin{theor}\label{thmain}  There is a unique pseudometric $\rho$ in~$C $ which restricts to the  metric~$d $ in $X\subset C$, restricts to the path pseudometric in every   component  of~$C$,    and   is minimal in the class of such pseudometrics in~$C$. For all $k\geq 1$,  we have $\rho\vert_{X_k}=d_k$.
  \end{theor}

  The minimality of $\rho$ means that for any pseudometric $\rho'$ in $C$ which restricts to~$d$ in~$X$ and to the path pseudometrics in the  components  of~$C$, we have   $\rho(a,b) \leq \rho'(a,b)$ for all $a,b \in C$. The uniqueness of such~$\rho$ is obvious;  we  only need to prove the existence and   the equality $\rho\vert_{X_k}=d_k$. We do it   in Section~\ref{proof1}.

 In the proof of Theorem~\ref{thmain} we will use a partial order in~$C$ defined as follows. We say that a path in~$C$ \emph{goes down} if whenever it enters an edge    connecting a point of $X_k$ to a point of $X_{k+1}$, it goes in this edge in the direction from $  X_k$ to $  X_{k+1}$.  We say that a point $b\in C$ lies \emph{below}   a point $a\in C  $   if $a \neq b$ and there is a path in~$C$ going down  from~$a$ to~$b$. 
  For example, for   $x\in X$ and integers $k < l$, the point $x_{(l)}\in X_l  $ lies below   $x_{(k)}\in X_k  $.

  \subsection{The metric space $\overline C$}\label{overlC} Consider  now the metric quotient $\overline C=(\overline C, \overline \rho)$ of the trimming cylinder  $(C=C(X), \rho)$. Since~$X$ is a metric space, the embedding  $X =X_0 \subset C$   induces a  metric embedding $X\hookrightarrow \overline C$. We   view~$X$ as a  subspace of~$\overline C$ via this embedding.

 We say that  a vertex $x_{(k)}\in X_k$ of~$C$  is \emph{special} if for all $l\geq k$, the edge  of~$C$ connecting $x_{(l)}$ to $x_{(l+1)}$ has zero length, i.e.,
 ${\underline {d_l}} (x_{(l)})=0$.
  A  component~$L$   of~$C$ is \emph{special} if it has at least one  special vertex.
   Then all special vertices of~$L$  and  all edges  between them form a  tree   which   collapses into a   single vertex  in $\overline L\subset \overline C$. We  call  the latter vertex   the \emph{root}  of~$\overline L$.

    We can now describe~$\overline C$. If   $\card (X_\infty)=1$, then $(C, \rho)$ is a pseudometric tree with the path pseudometric and $(\overline C, \overline \rho)$   is the associated metric tree with the path metric.
If $\card (X_\infty)\geq 2$,  then  $(\overline C, \overline \rho)$    is obtained from $(C, \rho)$ in two steps. First,  each component~$L$ of~$C$ is replaced with its metric quotient - the  metric  tree~$\overline L$ (which has a distinguished root if~$L$ is special).    Second, for all $u,v\in  X_\infty =\pi_0(C)$ such that $d_\infty (u,v)=0$ and
  both corresponding components $L_u, L_v$ of~$C$ are special, we identify (glue) the roots
 of  $\overline{L_u}$ and $ \overline {L_v}$. This gives a metric graph  whose   components are   trees. We endow this graph with the unique metric   such that the projection from~$C$ to this graph is distance preserving.  The  resulting metric space    is  $(\overline C, \overline \rho)$ as is clear from the definition of~$\rho$   in Section~\ref{proof2}.

 \subsection{Example} The trimming cylinders arising  in Examples~\ref{A  construction of rooted forests} and \ref{A ++ construction of rooted forests} have no special vertices.   If we modify  the assumptions  there   to $\delta_k(X_k) \subset (0,  \infty)$   for $k=0,1,..., N$ with some   $N\geq 0$ and $\delta_k=0$ for all $k>N$, then the  trimming sequence of  the metric space $(X_0, d_0)$ consists of the   metric spaces $\{(X_k, d_k)\}_{k=0}^{N}$, the set $X_{N+1} $ with zero pseudometric,  and  singletons corresponding to all $k>N+1$. All the vertices of the associated trimming cylinder belonging to $X_{N+1} $ are special.

  \section{The series~$\sigma$ and proof of Theorem~\ref{thmain}}\label{proof1}

 In this section, as above, $X=(X,d)$ is a metric space.

   \subsection{The series~$\sigma$}\label{The series} For each $x\in X$, consider the infinite series
  \begin{equation}\label{almotropdfdfdfds---} \sigma  (x)= \sum_{k=0}^{\infty} \,  \underline{d_k} (x_{(k)} )=\underline{d} (x ) + \sum_{k=1}^{\infty} \,  \underline{d_k} (x_{(k)} )  . \end{equation} Here    $(x_{(k)} \in X_k)_{k\geq 0}$ is the trimming sequence of~$x$ and $\underline{d_k}: X_k \to \RR_+$ is the   function induced by  the metric $d_k$  in $X_k$.  For  any $n \geq 1$,  set
   \begin{equation}\label{almotropdfdfdfds---part} \sigma^n  (x)= \sum_{k=0}^{n-1} \,  \underline{d_k} (x_{(k)} )
   = \underline{d} (x ) + \sum_{k=1}^{n-1} \,  \underline{d_k} (x_{(k)} ).   \end{equation}
  This gives a well-defined function $\sigma^n:X\to \RR_+$. In particular, $\sigma^1= \underline{d}$.


  \begin{lemma}\label{sigsig}  (i) For any distinct   $x,y\in X$ such that $x\cong y$, we have
\begin{equation}\label{almotrops41aaas} d(x,y)  = \sigma^{m}(x ) +\sigma^{m }(y ) \end{equation}
where $m=m(x,y) \geq 1 $ is the meeting index of~$x$ and~$y$;

(ii) For any   $x,y\in X$ such that $x\ncong y$,  both series $\sigma(x)$ and $\sigma (y)$ converge and \begin{equation}\label{almotrops41} d(x,y)  = d_{\infty}(x_{(\infty)},y_{(\infty)}) +\sigma(x )+ \sigma(y )  .\end{equation}
  \end{lemma}

 \begin{proof}   To prove (i), note that for all  $s <m$, we have  $x_{(s)} \neq y_{(s)}$ and therefore
 \begin{equation}\label{almotree} d_{s}(x_{(s)},y_{(s)})- d_{s+1}(x_{(s+1)},y_{(s+1)})= \underline{d_s} (x_{(s)} ) +\underline{d_s} (y_{(s)} ) . \end{equation}
Summing up   over   $s=0, 1, ..., m-1$ and using that $d_{m}(x_{(m)},y_{(m)})=0$,  we get
$$ d(x,y)  = d(x_{(0)},y_{(0)})   =\sigma^{m}(x ) +\sigma^{m}(y ).$$
 To prove (ii), note that   $x_{(s)} \neq y_{(s)}$ for all $s \geq 0$. Thus, we have  \eqref{almotree} for all $s \geq 0$. Summing up   over $s=0,1,..., n-1$, we obtain for all $n\geq 0$,
  \begin{equation}\label{almotrops} d(x,y)=d_{0}(x_{(0)},y_{(0)}) = d_{n}(x_{(n)},y_{(n)}) +\sigma^n  (x)+ \sigma^n(y). \end{equation}
   Therefore all partial sums of the series $\sigma(x)$ and $\sigma(y)$ are bounded above by $ d(x,y)  $. Hence, both series  converge. Taking the limit $n\to \infty$ in  \eqref{almotrops}, we obtain  \eqref{almotrops41}.
 \end{proof}


   \begin{lemma}\label{sigsig+} If $\card(X_\infty)\geq 2 $, then the series  $\sigma(x)$ converges for all $x\in X$.   \end{lemma}

This   directly follows from Lemma~\ref{sigsig}.(ii) as we can   pick   $y\in X $   so that $x\ncong y$. 

\subsection{Proof of Theorem~\ref{thmain}}\label{proof2} Suppose first that $\card(X_\infty)=1$. Then~$C$ is (path) connected and~$\rho$ is  the path pseudometric in~$C$. We verify   that   $\rho (x_{(k)}, y_{(k)})=d_k (x_{(k)}, y_{(k)})$ for all $x,y\in X$ and all $k\geq 0$. The connectedness of~$C$ implies that   $x \cong y$. Let $m=m(x,y)$ be the meeting index of~$x$ and~$y$. If $k\geq m$, then $x_{(k)}= y_{(k)}$ and $\rho (x_{(k)}, y_{(k)})=d_k (x_{(k)}, y_{(k)})=0$. If $k <m$, then an injective path from $x_{(k)}$ to $y_{(k)}$ in~$C$ is formed by the   edges  $$ e_{x_{(k)} }, e_{x_{(k+1)} },...,  e_{x_{(m-1)} }, e_{y_{(m-1)}},  ..., e_{y_{(k+1)}}, e_{y_{(k)}}.$$
Therefore
\begin{equation}\label{almotr--} \rho(x_{(k)},y_{(k)})     =\sum_{s=k}^{m-1} (\underline{d_s} (x_{(s)} )+ \underline{d_s} (y_{(s)} )) . \end{equation}
 By  the definition  of the metric $d_{s+1}$, for all $s < m$, we have \eqref{almotree}.
Summing up these equalities over $s=k, ..., m-1$, we  obtain that
$$ d_{k}(x_{(k)},y_{(k)})=d_{k}(x_{(k)},y_{(k)})- d_{m}(x_{(m)},y_{(m)})= \sum_{s=k}^{m-1} (\underline{d_s} (x_{(s)} )+ \underline{d_s} (y_{(s)} )) .  $$
 Comparing with \eqref{almotr--}, we obtain that
$ \rho(x_{(k)},y_{(k)})  = d_{k}(x_{(k)},y_{(k)})$. 

   Suppose now that  $\card(X_\infty)\geq 2$. By Lemma~\ref{sigsig+},  the series  \eqref{almotropdfdfdfds---} converges for all $x\in X$ and yields a function $\sigma: X\to \RR_+$.
   This function  extends to~$C  $ as follows: for any $a\in C$ lying in the edge  $e=e_{x_{(k)}}$  connecting the vertices $x_{(k)}$ and $x_{(k+1)}$ with $x\in X, k\geq 0$, set
   $$ \sigma(a)= d_e(a, x_{(k+1)})+\sum_{s=k+1}^{\infty} \,  \underline{d_s} (x_{(s)} ) $$
   where $d_e$ is the  pseudometric in~$e$ induced by the   fixed homeomorphism $e\approx [0,  {\underline {d_k} } (x_{(k)})]$ if ${\underline {d_k} } (x_{(k)})>0$ and $d_e=0$ if ${\underline {d_k} } (x_{(k)})=0$.
   The  infinite series  in the expression for $\sigma(a)$ is majorated by   $\sigma (x)$ and  therefore converges.
   In particular, $\sigma(x_{(k)})= \sum_{s\geq k}  \,  \underline{d_s} (x_{(s)} ) $ for all $x\in X$, $k \geq 0$. It is clear that the  function   $\sigma \colon C\to \RR_+$ is continuous. The next lemma implies that~$\sigma$ is monotonous in the sense that if $b\in C$ lies below $a\in C$, then $\sigma(b)\leq \sigma(a)$.

  \begin{lemma}\label{sigpointssig}   If   $a,b\in C$ lie in the same component~$L$ of~$C$, then
   \begin{equation}\label{ineqs} \vert \sigma (a) -\sigma (b) \vert \leq d_L(a,b) \leq \sigma (a) +\sigma (b)\end{equation}
   where $d_L$ is the  path pseudometric in~$L$. Moreover, if~$b $ lies below~$a $, then
    \begin{equation}\label{ineqs+} \sigma (a) =d_L(a,b)+\sigma (b). \end{equation}
  \end{lemma}

 \begin{proof} Formula \eqref{ineqs+} follows from the definitions.
 For $a=b$ the formula \eqref{ineqs} is obvious.  If one of the points $a,b$ lies below the other one,   then \eqref{ineqs} follows from \eqref{ineqs+}.   In all other cases, the injective path from~$a$ to~$b$  is V-shaped, i.e.,  goes down  from~$a$ to a certain vertex~$v$ of~$L$ and then goes up  to~$b$  (a  path   goes up  if the inverse path  goes down). Then $$\sigma(a)= d_L(a, v)+\sigma(v),\,\,\, \sigma(b)= d_L(b, v)+\sigma(v) , $$
   and $d_L(a,b)=d_L(a, v)+ d_L(b, v)$.
     This  easily implies \eqref{ineqs}. \end{proof}

      To complete the proof of the theorem, consider  the map $q: C\to  \pi_0(C)= X_\infty$ carrying each point    to its path connected component. For any $a,b \in C$, set
    \begin{equation*}
    \rho (a,b)=\left\{
                \begin{array}{ll}
                d_L(a,b) \quad {\text {if}} \,\,  a,b \,\,  {\text {lie in the same component}} \, \, L \,\, {\rm {of}}\,\, C , \\
                  d_\infty( q(a), q(b)  )+\sigma(a)+ \sigma(b) \quad {\text {if}} \,\,  q(a) \neq q(b).
                \end{array}
              \right.
\end{equation*}
We claim that the  map $\rho: C \times C \to \RR_+$ is a pseudometric. That   $\rho(a,a)=0$ and $\rho(a,b)=\rho(b,a)$ for all $a,b\in C$ is clear.  For points of~$C$ lying in the same component~$L$, the triangle inequality  follows from the one for $d_L$.
For points of~$C$ lying in three different components, the triangle inequality   follows from the  one for $d_\infty$. If  $a,b \in L $ for a component $L\subset C $ and $c\in C \setminus L$, then  by Lemma~\ref{sigpointssig},  $$\vert \rho(a,c)-\rho(b,c) \vert
= \vert \sigma (a) -\sigma (b) \vert \leq d_L(a,b) =\rho(a,b) $$
and $$  \rho(a,b)=d_L(a,b) \leq \sigma (a) +\sigma (b) \leq \rho(a,c)+\rho(b,c).$$

We check  that $\rho\vert_{X_k}=d_k$ for all~$k$. Pick any $x,y \in X$. If $x\cong y$, then the  arguments given in the case of connected $C$ apply and show that  $  \rho(x_{(k)},y_{(k)}) =d_{k}(x_{(k)},y_{(k)}) $. If $x\ncong y$, then similar arguments show that
for all $n\geq k$,
$$d_{k}(x_{(k)},y_{(k)})= d_{n}(x_{(n)},y_{(n)})+ \sum_{s=k}^{n-1} (\underline{d_s} (x_{(s)} )+ \underline{d_s} (y_{(s)} )). $$
Taking the limit   $n \to \infty$, we get
$$d_{k}(x_{(k)},y_{(k)}) =  d_\infty(  x_{(\infty)} ,  y_{(\infty)}   )+\sigma(x_{(k)})+ \sigma(y_{(k)}) =\rho(x_{(k)},y_{(k)}). $$

Finally, we prove that $\rho(a,b) \leq \rho'(a,b)$ for any $a,b \in C$ and any pseudometric~$\rho'$ in~$C$ which restricts to~$d$ in~$X$ and to the path pseudometrics in the   components  of~$C$.  It suffices to treat the case where $a,b$ lie in different components, say, $L, M$   of~$C$. Pick a path  in~$L$ starting in~$a $, going up,  and ending in  some   $x\in X$. Then $\sigma(x) =d_L(x,a) +\sigma(a) $. Similarly, pick a path  in~$M$ starting in~$b $, going up,  and ending in  some   $y\in X$. Then  $\sigma(y) =d_M(y,b) +\sigma(b)$.
We have $q(x)=q(a) \neq q(y)=q(b)$ and therefore  $$d(x,y)=\rho(x,y)= d_\infty( q(x), q(y)  )+\sigma(x)+ \sigma(y)$$
$$=
d_\infty( q(a), q(b)  )+d_L(x,a) +\sigma(a)+  d_M(y,b) +\sigma(b)$$
$$= \rho(a,b)+d_L(x,a) +  d_M(y,b).  $$
At the same time, the assumptions on $\rho'$ imply that
$$d(x,y)=\rho'(x,y) \leq \rho'(x,a) + \rho'(a,b) +\rho'(y,b)= \rho'(a,b) +d_L(x,a) +  d_M(y,b).$$
We conclude that
$$\rho(a,b)= d(x,y)- d_L(x,a) -  d_M(y,b) \leq \rho'(a,b) .$$

\subsection{Remark} If $\card(X_\infty)=1$, then the series $\sigma(x)$ may  converge or not. For instance, if in Examples~\ref{A  construction of rooted forests} and~\ref{A ++ construction of rooted forests} we set $ \delta_k(x)=1$ for all $x\in X,k\geq 0$, then $\sigma^n(x)=n$ for all~$n$ and  $\sigma(x)$ does not converge. Setting $\delta_k (x)=2^{-k}$ for all $x,k$, we obtain   examples where $\sigma (x) $ converges for all $x$. In general, if  $\sigma(x)$ converges for some $x\in X$, then it converges for  all $x\in X$. Indeed, the assumption $\card(X_\infty)=1$ ensures that any   $x,y \in X$ project to the same element of $X_k$ for all sufficiently big~$k$ and therefore the series $\sigma(x), \sigma(y)$ differ only in a finite number of terms.  Also, if (the only component of) $C$ is special, then   all terms of the series $\sigma(x)$  starting from a certain place are equal to zero and   $\sigma(x)$ converges.


 \section{The tight span versus the trimming cylinder}

  We recall the   tight span    following \cite{Dr},  \cite{DMT} and    relate it to the trimming cylinder. We also  discuss the tight spans of pseudometric spaces.

 \subsection{Tight span   of a metric space}\label{Basics} The \emph{tight span} of  a  metric space $(X,d)$ is the metric space $(T(X), d_T)$   consisting of all functions $f \colon X\to \RR_+$
such that
\begin{equation}\label{titi} f(x)=\sup_{y\in X}(d(x,y)-f(y)) \,\,\, {\text {for all}} \, \,\, x\in X.\end{equation}
This identity  may be restated by saying  that

$(\ast)$ $ f(x)+  f(y) \geq d(x,y)$ for all  $x,y \in X$ and

$(\ast \ast)$
 for any $x\in X$ and any real number $\varepsilon >0$, there is $ y \in X $ such that
 $    f(x)+   f(y) \leq d(x,y) +\varepsilon$.

 The   metric  $d_T$ in $T(X)$ is defined by    \begin{equation}\label{metric} d_T(f,g)=\sup_{x\in X}\vert f(x)-g(x)\vert \,\, {\text {for any}} \,\,f,g\in T(X) \end{equation}
 (here the set $\{\vert f(x)-g(x)\vert\}_{x\in X}$ is bounded above  and    has a well-defined supremum). The map  carrying any $x\in X$ to the function $X\to \RR_+,  y\mapsto d(x,y)$   is   a metric  embedding $X\hookrightarrow T(X)$.


Each      $f  \in T(X)$ is  minimal in the set of functions $f':X\to \RR_+$ satisfying $( \ast)$:  if  $f\geq f' $ (in the sense that $f(x)\geq f'(x)$ for all $x\in X$), then $f=f'$. Indeed,
 $$f(x) \geq f'(x) \geq \sup_{y\in X}(d(x,y)-f'(y)) \geq \sup_{y\in X}(d(x,y)-f(y))=f(x)$$
for all $x\in X$, and so  $f(x)=f'(x)$.


If $X$ is a singleton, then    $T(X)=\{0\}$. If $\card (X)\geq 2$, then   \eqref{titi} may be reformulated as follows (see \cite{Dr}, Sect.\ 1.4): for every $x\in X$,
\begin{equation}\label{titi+} f(x)=\sup_{y\in X\setminus \{x\}}(d(x,y)-f(y)).\end{equation}
This  identity     may be   restated by saying that

 $(\ast)'$  $ f(x)+  f(y) \geq d(x,y)$ for all distinct $x,y \in X$ and

$(\ast \ast)'$
 for any $x\in X$ and any real number $\varepsilon >0$, there is $ y \in X\setminus \{x\}$ such that
 $    f(x)+   f(y) \leq d(x,y) +\varepsilon$.

  We now     relate   the tight span $T(X)$ to the trimming cylinder $ C (X)$.


\begin{theor}\label{cylli}  For any   metric space $X=(X,d)$, there is a canonical distance preserving  map   $ C(X)\to T(X)$ extending the standard  embedding $X\hookrightarrow T(X)$.
  \end{theor}

 \begin{proof}    For  $a\in C=C(X) $,     define a  function
 $f=f_{a}:X \to \RR_+$  by $f(x)= \rho(x,a) $ for all $x\in X \subset C$ where $\rho=\rho_X$ is the pseudometric in~$C$.   We claim that $f\in T(X)$. Note that  for any  $x,y\in X$, the triangle inequality for~$\rho$ implies that $f(x)+f(y) \geq \rho(x,y)=d(x,y)$. Thus, $f$ satisfies Condition  $( \ast)$  above. Instead of Condition $( \ast\ast)$, we check a stronger claim:
for   any $x\in X$, there is $y\in X$ such that $f(x)+f(y) = d(x,y)$. In other words, we  find $y\in X$ such that~$a$ lies between~$x$ and~$y$ in~$C$ (see Section~\ref{Preliminaries} for  \lq\lq betweenness").
 Assume for concreteness that $a$ lies  in  an edge  of~$C$  connecting   vertices $v\in X_k$ and $p_k(v) \in X_{k+1}$  for some $k\geq 0$ (possibly, $a=v$). We   separate several cases.

 (i) Let  $x,v$ lie in different components of~$C$. Pick    $y\in X$ such that $y_{(k)}=v$.
 The definition of~$\rho$   shows that
 $$f(x)+   f(y) =\rho(x,a)+\rho(y,a)=\rho(x,a)+ \sigma(y)-\sigma(a)=\rho(x,y)=d(x,y)  .$$

 (ii) Let $x,v$ lie in the same component of~$C$ and 
 $x_{(k)} \neq v$. Pick   $y\in X$ such that $y_{(k)}=v$. Then $x\cong y$. Consider the injective path in~$C$ going from $y$ down to $y_{(k)}=v$, then further down to $y_{(m)}=x_{(m)}$ where $m=m(x,y) >k$ is the meeting index of~$x, y$, and then up to $x$. This path from~$y$ to~$x$  has length  $\rho(x,y)=d(x,y)$. Since the point~$a$ lies on this path, its length $d(x,y)  $  is   equal to $\rho(x,a)+\rho(y,a)=f(x)+   f(y)$.

 (iii) Let $x_{(k)}=v$ and $\card(X_k) \geq 2$. Pick   $y\in X$ such that $x\ncong y$. The same argument as in (i) shows that $    f(x)+   f(y) =d(x,y)  $.

 (iv) Let $\card(X_k)=1$, i.e., $X_k=\{v\}$. Then the   edge containing~$a$ (i.e., the edge from $v$ to   $p_k(v)$) has zero length, and   $f=\rho(-, a)=\rho(-,v)$. Thus, it is enough to treat the case $a=v$.   If $\card(X_{k-1})\geq 2$, then there is $y\in X$ such that $x_{(k-1)} \neq y_{(k-1)}$. The injective path from $y$ to~$x$ in~$C$ must pass  through the only element~$v$ of $X_k$ and so  $    f(x)+   f(y) =d(x,y)  $.  If $\card(X_{k-1})=1$, then   $f= \rho(-,u)$ where~$u$ is the only element of $X_{k-1}$. Proceeding by induction, we eventually find some $l<k$ such that $\card(X_l) \geq 2$ and the argument above works or deduce that $\card(X)=1$ in which case   the theorem is obvious.

 We next verify   that $d_T(f_a, f_{b})= \rho(a,b)$ for all $a,b\in C$. The case $a=b$ is obvious as  both sides are equal to zero. Assume that $a\neq b$ and recall that $$d_T(f_a, f_{b})=\sup_{x\in X}\vert f_{a}(x)-f_{b}(x)\vert=\sup_{x\in X}\vert \rho(x,a)-\rho(x,b)\vert.$$ The triangle inequality for~$\rho$ yields $d_T(f_a, f_{b}) \leq \rho(a,b)$. To prove that this  is an equality, we   find a point $x\in X\subset C$ such that either $b$ lies between~$a$ and~$x$ (with respect to the pseudometric~$\rho$ in~$C$) or $a$ lies between~$b$ and~$x$. In  both cases $ \vert \rho(x,a)-\rho(x,b) \vert=\rho(a,b)  $. Note that going up from $a,b$ in~$C$ we eventually hit certain elements, respectively, $x,y \in X\subset C$.  If $a,b$ belong to   different components of~$C$, then $a$ lies between~$b$ and~$x$ as  follows from the definition of~$\rho$.
 Assume   that $a,b$ belong to the same component~$L$ of~$C$. If~$b$ lies below~$a$, then~$a$ lies between~$b$ and~$x$ as follows  from the definition of the path pseudometric   in~$L$.
In all other cases,  an injective path from~$y$ to~$a$ in $L$ necessarily passes by $b$ and therefore~$b$ lies between~$a$ and~$y$.  \end{proof}


\begin{corol}\label{cyllicorol}  For any   metric space~$X$, there is a canonical metric embedding  $ {\overline {C (X)}} \hookrightarrow T(X)$  extending the standard  embedding $X\hookrightarrow T(X)$.
  \end{corol}

The  embedding  $\overline C= {\overline {C (X)}} \hookrightarrow T(X)$  is induced by the distance preserving map   from   Theorem~\ref{cylli}.  We  will identify ${\overline {C  }}$ with its image in $ T(X)$ under this embedding, i.e., view ${\overline {C  }} $ as a metric subspace of $ T(X)$.

\subsection{The tight span of a pseudometric space}
The definition of the tight span via \eqref{titi}   extends word for word to  pseudometric spaces. This however does not give new metric spaces because, by the next lemma, a  pseudometric space and   its metric quotient have the same   tight span.

 \begin{lemma}\label{psedotight11} Let $X=(X,d)$ be  a pseudometric space and let $q:X \to \overline X = X/{\sim_d }$ be the  projection from~$X$ to its metric quotient. Then the formula ${h}\mapsto {h}  q$, where~${h}$ runs over $ T(\overline X)$, defines an isometry   $T(\overline X) \to T(X)$.
  \end{lemma}

 \begin{proof} Let $\overline d$ be the metric in $\overline X$ induced by~$d$. We pick any function ${h} \colon \overline X \to \RR_+$ in $  T(\overline X)$ and verify Conditions $(\ast)$, $(\ast\ast)$ for ${h}q\colon X \to \RR_+$.  For any $x,y \in X$, we have ${h}q(x)+ {h}q(y) \geq d(x,y)$ because  if $q(x)=q(y)$, then $d(x,y)=0$, and if $q(x)\neq q(y)$,  then
 ${h}q(x)+ {h}q(y) \geq \overline d (q(x), q(y))=d(x,y)$. To verify $(\ast\ast)$, pick any $x\in X$ and $\varepsilon>0$. Since $  {h} \in T(\overline X)$ and $q$ is onto, there is $y\in X$ such that $${h}q(x)+ {h}q(y) \leq \overline d (q(x), q(y))+\varepsilon =d(x,y)+\varepsilon .$$
 Thus,   ${h}q \in T(X)$.
 That   the map $T(\overline X) \to T(X), {h}\mapsto {h}  q$ is injective and metric preserving is clear from the definitions. To  prove   surjectivity,   we show that each function $f\in T(X)$ takes equal values on   $\sim_d$-equivalent points of~$X$. We have $f(y)-d(x,y) \leq f(x)$ for all $x, y \in X$ because (cf.\ \cite{DMT}, Section 2)
$$f(y) -d(x,y)= \sup_{z\in X}(d(y,z)-f(z))- d(x,y)$$
$$=\sup_{z\in X}(d(y,z)-f(z) - d(x,y)) \leq \sup_{z\in X}(d(y,x)+d(x,z)-f(z) - d(x,y))$$
$$=\sup_{z\in X}(d(x,z)-f(z))= f(x).$$
If $d(x,y)=0$, then we get $f(y)  \leq f(x)$. Exchanging $x,y$, we get   $f(x)  \leq f(y)$. Thus, if $d(x,y)=0$, then $f(x)=f(y)$. As a consequence, $f=hq$ for a function ${h} \colon \overline X \to \RR_+$. Conditions $(\ast)$, $(\ast\ast)$ for~$f $   imply the same conditions for~$h $.
 \end{proof}

 \section{Trimming versus the tight span}

We discuss further  relations between trimming and the tight span.
In   this  section,   $X=(X,d)$ is a   metric space.

 \subsection{The canonical embedding}  Recall the function
 $\underline d: X \to \RR_+$ induced by~$d$.

  \begin{theor}\label{tight11} There is a canonical metric embedding $T(t(X)) \hookrightarrow T(X)$ whose image consists of all   $f\in T(X)$ such that $f \geq {\underline d}$.
  \end{theor}

 \begin{proof} Let $t(X)=(X_1, d_1)$ as in Section~\ref{Trimming}.   Assume  first that $X_1$ is a singleton. Then $T(X_1)=\{0\}$  and the embedding $T(X_1) \hookrightarrow T(X)$ carries $0$ to $\underline d$. We need to prove that ${\underline d} \in T(X)$.  If $\card (X) =1$, then ${\underline d} =0 \in T(X)$.   If $\card (X) \geq 2$, then our assumption on $X_1$  implies that $d(x,y)={\underline d}(x) +{\underline d}(y)$ for all distinct $x,y \in X$. Applying \eqref{titi+}   to $f={\underline d}$, we deduce that $ {\underline d} \in T(X)$. By the minimality of any  $f\in T(X)$, if $f \geq \underline d$, then  $f=\underline d$. This gives the second claim of the theorem.

 Assume now that $\card(X_1) \geq 2$, and    let $p\colon X\to X_1$ be the trimming projection.   Given $ g\in T(X_1)$, we define a  function $\hat g\colon X\to \RR_+$   by
  $$\hat g (x)= gp(x)+ {\underline d}(x) =g(x_{(1)}) +  {\underline d}(x)$$
  for all $x\in X$.  We claim that  $\hat g \in T(X)$.
  To see it, we check Conditions $(\ast  )'$ and $(\ast \ast)'$ of Section~\ref{Basics} for $f=\hat g$.
   To check $(\ast  )'$,   we separate two cases:  $p(x)=p(y)$ and $p(x)\neq p(y)$.
 If $p(x)=p(y)$, then
$(\ast  )'$ holds because
 $$\hat g(x)+ \hat g(y) \geq {\underline d}(x) + {\underline d}(y) =d(x,y).$$
  If $p(x)\neq p(y)$, then $(\ast  )'$ holds because
$$\hat g(x)+ \hat g(y) = gp(x)+ gp(y)+{\underline d}(x) + {\underline d}(y) $$
$$\geq d_1(p(x), p(y))
+{\underline d}(x) + {\underline d}(y)= d(x,y)$$ where the inequality   follows from the inclusion  $g\in T(X_1) $ and the final equality holds by the definition of  $d_1$.
To verify   $(\ast \ast )'$,  pick any $x\in X$ and   $\varepsilon >0$. Using $(\ast \ast )'$ for $g\in T(X_1)$ and using the surjectivity of  $p \colon X \to X_1$, we obtain
an element $ y \in X$ such that $p(x)\neq p(y)$ and
 $$ d_1(p(x),p(y)) +\varepsilon \geq gp(x) +gp(y).$$
Then $ x \neq y$ and $(\ast \ast )'$ for $\hat g$ follows: $$d(x,y)+\varepsilon =  d_1(p(x) ,p(y)) +   {\underline d}(x) + {\underline d}( y) +\varepsilon
 $$  $$\geq   gp(x) +gp(y)+ {\underline d}(x) + {\underline d}( y) =\hat g(x)+ \hat g(y). $$
 The   map  $T(X_1)\to T(X), g\mapsto \hat g$ is a metric embedding: for any $g ,h \in T(X_1)$,
$$d_T(\hat g , \hat h)= \sup_{x\in X}\vert \hat g (x)-\hat h(x)\vert
=  \sup_{x\in X}\vert   g p(x)-  hp(x)\vert$$
$$=\sup_{z\in X_1}\vert   g (z)-  h(z)\vert
=(d_1)_T(  g ,   h) .$$
 To show that any  $f\in T(X)$ satisfying $f \geq \underline d$ lies in the image of our embedding, set $f_-=f- \underline d\colon X \to \RR_+$ and recall the pseudometric $d^-$ in~$X$ defined in Section~\ref{Trimming}. Since $\card(X) \geq \card (X_1)\geq 2$, we have by \eqref{titi+} that for any $x\in X$,
 $$f_-(x)= f(x) - {\underline d}(x)=\sup_{y\in X\setminus \{x\}}(d(x,y)- {\underline d}(x) -  {\underline d}(y)+  {\underline d}(y)-f(y))$$
 $$=  \sup_{y\in X \setminus \{x\}}(d^{-} (x,y) - f_-(y) )=  \sup_{y\in X  }(d^{-} (x,y) - f_-(y) ) . $$
The last equality holds because its left-hand side is non-negative (being equal to $f_-(x)$) while $d^{-} (x,x) - f_-(x)=  - f_-(x)\leq 0$.  Thus, $f_- \in T(X, d^{-})$. By Lemma~\ref{psedotight11},  $f_-=gp$ for some $g\in T(X_1)$. Then $f=f_-+ \underline d=\hat g$.
 \end{proof}


\subsection{The trimming filtration}\label{The trimming filtration} From now on, we  view $T(t(X))=T(X_1)$ as a metric subspace of $T(X)$ via the  embedding from  Theorem~\ref{tight11}. Iterating the inclusion   $T(X) \supset T(X_1)$, we obtain a filtration
\begin{equation}\label{filt} T(X) \supset T(X_1) \supset T(X_2) \supset T(X_3) \supset \cdots  \end{equation}
 where   $t^n(X)= (X_n, d_n)$ for all $n \geq 1$.  In   \eqref{filt}, we identify  each $T(X_n)$   with its image under the metric embedding  $T(X_n) \hookrightarrow T(X)$ obtained as the composition of the  embeddings $T(X_n) \hookrightarrow T(X_{n-1}) \hookrightarrow  \cdots \hookrightarrow  T(X)$. This composition carries any   $f\in T(X_n)$ to the function $\hat f\colon X\to \RR_+$ defined by $\hat f( x)= f( x_{(n)})+ \sigma^{n}  (x ) $ for $x\in X$.
  Clearly,    $ \hat f \geq \sigma^{n}   $. An   induction on~$n$ deduces from Theorem~\ref{tight11} that    $$T(X_n)=\{g \in T(X)\, \vert \, g  \geq \sigma^{n}\} \subset T(X).$$

\begin{theor}\label{tight11coroll} All  terms of the filtration \eqref{filt} are closed subsets of $T(X)$.
  \end{theor}

 \begin{proof} It suffices to show that the set $T(X_1)\subset T(X)$ is   closed in $T(X)$ and to apply induction. We prove that the complementary set $U=T(X) \setminus T(X_1)$ is open in $T(X)$. By Theorem~\ref{tight11}, for any $f \in U$, there is  $a\in X$ such that $0 \leq f(a) <{\underline d}(a)$.   We claim that the open ball $B\subset T(X)$ with center~$f$ and radius $r= {\underline d}(a)- f(a)  >0 $  is contained in~$U$. Indeed, if $g\in B$, then $$r>d_T(f,g)=\sup_{x\in X}\vert f(x)-g(x)\vert\geq g(a)-f(a).$$
 Therefore $g(a)< f(a)+r ={\underline d}(a)$. Thus,    $g \ngeq \underline d$ and so $g\in U$. We conclude that~$U$ is open   and   $T(X_1) $ is   closed.
 \end{proof}

 \subsection{The metric space $\tau$} Consider the set
 $$\tau=\tau(X)=\cap_{n\geq 1} \, T(X_n) \subset T(X) $$
and endow it with the metric  obtained by restricting the one in $T(X)$.


   \begin{theor}\label{Xinfinity}  If either $\card(X_\infty) \geq 2$ or $\card(X_\infty) =1$  and the series  $\sigma(x)$ converges for all $x\in X$, then there is an isometry   $
   \tau \approx T(X_\infty)$. If $\card(X_\infty) =1$  and the series  $\sigma(x)$ does not converge  for some $x\in X$, then $\tau=\emptyset$.
  \end{theor}

 \begin{proof}  Recall that $X_\infty =X/\cong$ where $\cong $ is the equivalence relation in~$X$ defined in Section~\ref{Trimming}. We start with the case $\card(X_\infty) \geq 2$. By  Lemma~\ref{sigsig+},   the series $\sigma(x)$ converges for all $x\in X$ and yields a function $\sigma \colon X \to \RR_+$.
For any function $f:   X_\infty \to \RR_+$, we define a function $\tilde f : X\to \RR_+$ by $\tilde f (x) = f(x_{(\infty)}) +\sigma(x)$. We claim that if $f\in T(X_\infty)$, then $\tilde f \in T(X)$.
  Note that $\card(X)\geq \card (X_\infty) \geq 2$. Therefore to verify the inclusion $\tilde f \in T(X)$, it suffices to verify for $\tilde f$ Conditions $(\ast)'$ and $(\ast \ast)'$ from Section~\ref{Basics}. To verify $(\ast)'$, pick any distinct $x,y \in X$. If $x\cong y$, then Formula~\eqref{almotrops41aaas} implies
  that \begin{equation}\label{sdf} d(x,y) \leq \sigma(x) +\sigma(y) \leq \tilde f(x) +\tilde f (y). \end{equation} 
  If $x\ncong y$, then $x_{(\infty)} \neq y_{(\infty)}$.  Formula~\eqref{almotrops41} and Condition  $(\ast)'$ on~$f$ give
  $$d(x,y) \leq f(x_{(\infty)}) + f(y_{(\infty)}) +\sigma(x) +\sigma(y)= \tilde f(x) + \tilde f(y).$$
  To verify $(\ast\ast)'$, pick any   $x  \in X$ and $\varepsilon >0$. Since  $f: X_\infty \to \RR_+$ satisfies $(\ast\ast)'$ and the projection $X\to X_\infty$ is onto, there is $y\in X$ such that $x_{(\infty)} \neq y_{(\infty)}$ and $$    f(x_{(\infty)})+   f(y_{(\infty)}) \leq d_\infty (x_{(\infty)}, y_{(\infty)}) +\varepsilon.$$ Then  $x\ncong y$ and by \eqref{almotrops41},
  $$\tilde f(x) + \tilde f(y) \leq   d_\infty (x_{(\infty)}, y_{(\infty)}) +\varepsilon +\sigma(x) +\sigma(y)=d(x,y) +\varepsilon.$$
Thus, $\tilde f \in T(X)$. The     map  $T(X_\infty)\to T(X), f\mapsto \tilde f$ is a metric embedding:
$$d_T(\tilde f_1, \tilde f_2)= \sup_{x\in X}\vert \tilde f_1(x)-\tilde f_2(x)\vert
=  \sup_{x\in X}\vert     f_1(x_{(\infty)})-  f_2(x_{(\infty)})\vert$$
$$=\sup_{z\in X_\infty}\vert   f_1(z)-  f_2(z)\vert
=(d_\infty)_T(  f_1,   f_2) $$
for any $f_1,f_2 \in T(X_\infty)$.
Also,  for all $n\geq 1$, the obvious inequalities $\tilde f \geq \sigma \geq \sigma^n$  imply that $\tilde f \in T(X_n)\subset T(X)$. So, $\tilde f \in \tau$.

It remains to show  that for each   $g \in \tau$ there is $f\in T(X_\infty)$ such that $g=\tilde f$.  By Theorem~\ref{tight11}, the function $g-\sigma^1=g-\underline d $ is the composition of the projection $X\to t(X)$ and a certain  function $g_1\in T(X_1)$. Proceeding by induction, we deduce that for each $n\geq 1$, the function $g-\sigma^n $ is the composition of the projection $X\to X_n$ and a    function $g_n\in T(X_n)$. Thus,  $g \geq \sigma^n$ for all $n \geq 1$ and therefore $g\geq \sigma$. We check now  that the function $g-\sigma \geq 0$ is the composition of the projection $X\to X_\infty$ and a    function   $X_\infty \to \RR_+ $. If two points $x,y \in X$ project to the same point of $X_\infty$, then   there is an integer $n\geq 0$ such that $x_{(n)}=y_{(n)}$. Then      $$(g - \sigma^n)(x) = g_n(x_{(n)})=g_n(y_{(n)})=(g - \sigma^n)(y).$$ Also,  $x_{(s)}=y_{(s)}$   for all  $s\geq n$ and so  $$(\sigma-\sigma^n) (x)= \sum_{s\geq n}  \underline{d_s} (x_{(s)} )=  \sum_{s\geq n}  \underline{d_s} (y_{(s)} )=(\sigma-\sigma^n) (y). $$
Therefore,  $(g- \sigma)(x) =  (g- \sigma)(y)$. Thus, the function $g-\sigma \geq 0$
is the composition of the projection $X\to X_\infty$ and a    function $f: X_\infty \to \RR_+$. So, $g=\tilde f$. We claim that $f \in T(X_\infty)$. To see it, we verify that $f$ satisfies Conditions
$(\ast)$ and $(\ast \ast)$ from Section~\ref{Basics}. To verify $(\ast)$, pick any   $x,y \in X$. If  $x_{(\infty)} \neq y_{(\infty)} $, then
  $$f(x_{(\infty)} )+f(y_{(\infty)} )= g(x) - \sigma (x) + g(y)-\sigma(y) $$
$$\geq d(x,y) - \sigma (x)  -\sigma(y)=  d_{\infty}(x_{(\infty)},y_{(\infty)})$$
where  the inequality follows from the assumption $g\in T(X)$ and the last equality holds by \eqref{almotrops41}. If   $x_{(\infty)} = y_{(\infty)} $, then
  $$f(x_{(\infty)} )+f(y_{(\infty)} ) \geq 0=  d_{\infty}(x_{(\infty)},y_{(\infty)}).$$ To verify $(\ast\ast)$, pick any   $x  \in X$ and $\varepsilon >0$. Since $g\in T(X)$, there is $y\in X$ such that  $    g(x)+   g(y) \leq d(x,y) +\varepsilon$. Then
$$f(x_{(\infty)} )+f(y_{(\infty)} )= g(x) - \sigma (x) + g(y)-\sigma(y)  \leq d(x,y) +\varepsilon - \sigma (x)  -\sigma(y).$$
If  $x_{(\infty)} \neq y_{(\infty)} $, then the right-hand side is equal to $d_{\infty}(x_{(\infty)},y_{(\infty)}) +\varepsilon$ and we are done.
If   $x_{(\infty)} = y_{(\infty)} $, then the right-hand side is  smaller than or equal to $\varepsilon= d_{\infty}(x_{(\infty)},y_{(\infty)}) +\varepsilon$ as follows from
\eqref{sdf}.

Assume  that $\card(X_\infty) =1$, i.e., that $X_\infty$ is a singleton. If  $\tau \neq \emptyset$, then any $f\in  \tau=\cap_{n\geq 1}\,  T^n(X) $   majorates all the functions $\{  \sigma^{n} \}_n$ on~$X$.    Therefore    the series   $\sigma(x)$ converges for all $x\in X$. This  implies the last claim of the theorem.   Suppose now that   the series  $\sigma(x)$ converges for all $x\in X$. It defines then a function $\sigma:X \to \RR_+$. Below we prove that
  $\sigma \in T(X)$. This will imply the claim of the theorem. Indeed, since $\sigma\geq \sigma^n$ for all   $n\geq 1$, we   have then $\sigma \in \tau $. Any $ f\in \tau  $ satisfies   $f\geq \sigma^n$ for all ${n\geq 1}$ and so $f \geq \sigma$.  By the minimality of $f\in T(X)$ (see Section \ref{Basics}), the   inequality  $f \geq \sigma$ implies that $f=\sigma$.  Thus,   $ \tau=\{\sigma\}$ is isometric to $T(X_\infty)=\{0\}$.  To prove the inclusion $\sigma \in T(X)$, we verify   Conditions $(\ast)$ and $(\ast \ast)$ from Section~\ref{Basics}. The distance $d(x,y)$ between any points $x,y \in X$ may be computed  from the equality  $d(x,y)=\rho(x,y)$, where~$\rho$ is the pseudometric in the trimming cylinder of~$X$, and the expression~\eqref{almotr--} for $\rho(x,y)$ where $k=0$ and $m=m(x,y)$.  This gives $$d(x,y)=\sigma^{m}(x)+\sigma^{m}(y) \leq \sigma(x) +\sigma(y)$$
 which is   Condition~$(\ast)$ for~$\sigma$.  To check $(\ast \ast)$ for~$\sigma$, pick any $x\in X$ and $\varepsilon >0$. The assumption $\card(X_\infty) =1$ ensures that for every $y\in X $,  there is an integer  $k  \geq 0$ such that $x_{(k)}=y_{(k)}$. Let $k_y$ be the smallest such integer. If the   set of  integers $\{k_y \geq 0\}_{y\in X}$ is finite, then either $\card (X)=1$ or there is an integer $K\geq 1$ such that $\card (X_{K-1})\geq 2$ and  $\card (X_K)=1$. If $\card (X)=1$,  then   $\sigma=0\in T(X)=\{0\}$.  If $\card (X_{K-1})\geq 2$ and $\card (X_K)=1$, then $\sigma=\sigma^{K}$ and  there is $y\in X\setminus \{x\}$, such that $y_{(K-1)} \neq x_{(K-1)}$ and $y_{(K)} = x_{(K)}$. Then $$\sigma(x)+ \sigma(y)=\sigma^{K} (x) +\sigma^{K} (y)=\rho(x,y)=d(x,y)\leq d(x,y) +\varepsilon.$$
 If   the   set   $\{k_y\}_{y\in X}$ is infinite, then we can find $y\in X\setminus \{x\}$ with $k_y$ so big that $\sigma(x) - \sigma^{k_y}(x)<\varepsilon/2$. The equality $y_{(k_y)} = x_{(k_y)}$ implies   that
 $$\sigma(y) - \sigma^{k_y}(y)= \sigma(x) - \sigma^{k_y}(x) <\varepsilon/2 .$$
 Then $$ \sigma (x)+ \sigma (y) < \sigma^{k_y}(y)+ \sigma^{k_y}(x) +\varepsilon =\rho(x,y) +\varepsilon= d(x,y) +\varepsilon.$$ Here the   equality in the middle follows from the definition of the path pseudometric  $\rho$ and the conditions $y_{(k_y-1)} \neq x_{(k_y-1)}$, $ y_{(k_y)} = x_{(k_y)}$. Thus, Condition~$(\ast \ast)$ also holds for~$\sigma$. So, $\sigma \in T(X)$. \end{proof}

 \section{Main theorem}

 \subsection{Main theorem} We     state our main result  on  the tight span $T(X)$  of a metric space $X=(X,d)$. Recall the subspaces $\overline C$ and $\tau$ of   $T(X)$. Recall   the special components of the trimming cylinder $C=(C, \rho)$ and the  roots of their projections  to~$\overline C$. The set of   these roots is denoted   ${\overline C}_\bullet$.

\begin{theor}\label{cyllibb+}  For any metric space $X=(X,d)$,
 $$T(X)= \tau \cup {\overline C } \quad {\text and} \quad  \tau \cap {\overline C }= {\overline C }_\bullet.$$
  \end{theor}

  \begin{proof} We first consider the case where $\card(X_\infty) \geq 2$. Consider the distance preserving map $C\to T(X), a \mapsto f_a$ from  Theorem~\ref{cylli}. To prove the equality $T(X)= \tau \cup {\overline C }$, it suffices to show
  that each  function $f:X\to \RR_+$ belonging to $  T(X) \setminus \tau$   is equal to $f_a$ for some  $a\in C $. We prove a   stronger claim: $f=f_a$   for some $a\in C$ such that $\sigma(a)>0$ where $\sigma \colon C\to \RR_+$ is the function defined in the proof of Theorem~\ref{thmain}. Suppose first that  $f \notin  T(X_1)$. By Theorem~\ref{tight11}, there is $y\in X$ such that $0 \leq f(y) < {\underline d}(y)$.  The   edge of~$C$  connecting $y=y_{(0)}$ to $y_{(1)}$ has length   ${\underline d}(y)$ and   contains  a (unique) point~$a$
   such that $ \rho(y,a)= f(y) $.    Then $$\sigma(a)\geq \rho(a, y_{(1)})= {\underline d}(y) -f(y) >0$$ and  $f_{a}(y)=\rho(y,a)=f(y) $. For any $x\in X\setminus \{y\}$, Condition ($\ast$) in Section~\ref{Basics} and the fact that~$a$ lies between~$x$ and~$y$ in~$C$ imply that
   $$f(x) \geq d(y,x) - f(y) =\rho(y,x) - \rho(y,a)= \rho(x,a)=f_{a}(x).$$ Thus, $f\geq  f_{a}$.
    Since $f, f_{a } \in T(X)$, we conclude that   $f=f_{a }$.

    Suppose now that $ f\in  T(X_n) \setminus T(X_{n+1})$ with $n\geq 1$. By Section~\ref{The trimming filtration}, the embedding $T(X_n) \hookrightarrow T(X)$ carries any function  $ g: X_n \to \RR_+$ in $ T(X_n)$   to the function ${\hat g}:X \to \RR_+$ defined by ${\hat g}( x) = g(x_{(n)}) +\sigma^{n}(x)$ for all $x\in X$. Since $f\notin T(X_{n+1})=T(t(X_n)) \subset T(X_n)$,    Theorem~\ref{tight11} implies that there is $y\in X_n$ such that $0\leq f(y) < {\underline {d_n}}(y)$. The   edge of~$C$ connecting $y\in X_n$ to the projection of~$y$ to $X_{n+1}$  has length ${\underline {d_n}}(y)$ and   contains  a (unique) point~$a$
   such that $ \rho(y,a)= f(y) $. As above,  $\sigma(a)\geq {\underline {d_n}}(y) -f(y) >0$.   For any $x\in X$ such that $x_{(n)}=y$, the point~$y$ lies between~$x$ and~$a$ in~$C$, and therefore $${\hat f}(x)= f(y) +\sigma^{n}(x)=\rho(y,a)+\rho(x,y)=\rho(x,a)= f_a(x). $$    For any $x\in X$ such that $x_{(n)}\neq y$, the point $a$ lies between $x$ and~$y$ in~$C$. The inclusion $f\in T(X_n)$ and
  Condition ($\ast$) of Section~\ref{Basics} yield $${\hat f}(x)= f(x_{(n)}) +\sigma^{n}(x)\geq d_n(x_{(n)}, y) -f(y) +\sigma^{n}(x)=$$
  $$=\rho(x_{(n)}, y) -\rho(y,a) +\rho(x,x_{(n)})
  \geq \rho(x,y)-\rho(y,a)=\rho(x,a)= f_a(x) .$$   So, ${\hat f}\geq  f_{a}$ and,
    since ${\hat f}, f_{a } \in T(X)$, we  have   ${\hat f}=f_{a }$. 
     
     The arguments above prove the inclusion $T(X) \setminus \tau \subset {\overline C } \setminus \overline C_\bullet $. As a consequence, $T(X)= \tau \cup  {\overline C }  $
    and   $  \overline C_\bullet \subset \tau$.
It remains only to   show   that   $\tau \cap (\overline  C\setminus \overline C_\bullet )=\emptyset  $, i.e.,  that $f_a  \notin \tau  $ for  any   $a\in C$ such that $\sigma(a)>0$. Assume for concreteness that  $a=x_{(k)}$  or~$a$ lies inside the edge  of~$C$ connecting the vertices $x_{(k)}$, $x_{(k+1)}$ where $x\in X, k\geq 0$.
Let $L\subset C$ be the   component of~$C$ containing~$x$ and~$a$. The condition $\sigma(a)>0$ implies that there is $n \geq k$ such that $\underline {d_{n}} (x_{(n)})>0$. The  definition of the path pseudometric $d_L$ in~$L$    implies that   $d_L (a,x_{(n+1)})  \geq \underline {d_{n}} (x_{(n)}) >0$. It is also clear that~$a$ lies between the points~$x$ and  $x_{(n+1)}$ of~$L$. Then
 $$\sigma^{n+1}(x) = d_L (x, x_{(n+1)})=d_L(x,a) + d_L(a,x_{(n+1)}) > d_L(x,a)=\rho(x,a)=f_a(x).$$
We conclude that the function $f_a$ does not majorate the function $\sigma^{n+1}:X \to \RR_+$. By Section~\ref{The trimming filtration}, $f_a\notin T(X_{n+1} )\subset T(X)$.
Therefore, $f_a  \notin  \tau$.

Suppose now that $\card(X_\infty) = 1$ so that~$C$ is (path) connected. If the series $ \sigma(x)$ converges for all $x\in X$ (and   defines a function $\sigma:X\to \RR_+$), then   the theorem is proved exactly as in the case $\card(X_\infty)  \geq 2$.  If  the series $ \sigma(x)$ does not converge  for some $x\in X$, then~$C$ is non-special and ${\overline C }_\bullet =\emptyset$. The equality $T(X)=\tau \cup \overline C$  is proved then as in the case   $\card(X_\infty)  \geq 2$ suppressing all references to ${\overline C }_\bullet, \sigma$. Also,     $\tau= \emptyset$ by Theorem~\ref{Xinfinity} and so $ \tau \cap {\overline C }=\emptyset = {\overline C }_\bullet$. \end{proof}

\subsection{Remarks} 1. When~$C$ is (path) connected, i.e.,  $\card(X_\infty) = 1$, we can explicitly compute the sets~$\tau$ and ${\overline C }_\bullet$.
We consider three   cases:

(i) (the only component of) $C$ is special, or

(ii) $C$ is not special but the series $ \sigma(x)$ converges for all $x\in X$ and defines a function $\sigma:X\to \RR_+$, or

(iii) the series $ \sigma(x)$ does not converge  for some $x\in X$.

In the cases (i) and (ii),  the proof of Theorem~\ref{Xinfinity} shows that  $\tau=\{\sigma\}$. In the case (iii), the same theorem gives  
$\tau =\emptyset$. The set $ \overline C_\bullet$ is computed from the definitions using the connectedness of~$C$. Namely, $ \overline C_\bullet =\{\sigma\} $ in   the case (i) and    $\overline C_\bullet= \emptyset$ in the cases (ii) and (iii).

2. For  $x\in X $, let $L_x$ be the connected    component  of the trimming cylinder $C=C(X)$ containing   $x\in X=X_0\subset C$. If the pseudometric tree $L_x$ is special, then   the root of its metric quotient $\overline {L_x}    $  belongs to the set ${\overline C }_\bullet \subset  \tau  $. Comparing the canonical   embedding $\overline C\hookrightarrow T(X)$ with the isometry $T(X_\infty)=\tau  $, one can check that  the root in question is  the point $x_{(\infty)} \in X_\infty \subset T(X_\infty)=\tau $.

                     \end{document}